\documentclass[12pt]{amsart}

\usepackage{amssymb}
\usepackage{bm}
\usepackage{amsfonts}%This is to have nice letters
\usepackage{amsthm}
\usepackage{graphicx}
\usepackage{psfrag}
\usepackage{color}
\usepackage{url}
\usepackage{enumerate}

\newtheorem{thm}{Theorem}
\newtheorem{lemma}[thm]{Lemma}

\newtheorem{cor}[thm]{Corollary}

\theoremstyle{definition}

\newtheorem{remark}[thm]{Remark}

\numberwithin{equation}{section}

\newcommand{\NC}{{Q}}
\newcommand{\Q}{{\mathbb Q}}
\newcommand{\R}{{\mathbb R}}
\newcommand{\RR}{{R}}
\newcommand{\Z}{{\mathbb Z}}

\newcommand{\be}{\begin{eqnarray}}
\newcommand{\bea}{\begin{eqnarray*}}

\newcommand{\ee}{\end{eqnarray}}
\newcommand{\eea}{\end{eqnarray*}}
\newcommand{\gap}{\mathrm{Gap}}

\newcommand{\frob}{g}
\newcommand{\ffrob}{f}

\newcommand{\ve}{\boldsymbol}

\newcommand{\A}{A}
\newcommand{\modulo}{\,\mathrm{mod}\;}

\begin{document}%%%%%%%%%%%%%%%%%%%%%%%%%%%%%%%%%%%%%%%%%%%%%%%%%%%%%%%%
%%%%%%%%%%%%%%%%%%%%%%%%%%%%%%%%%%%%%%%%%%%%%%%%%%%%%%%%%%%%%%%%%%%%%%%%

\mbox{}
%\vspace{-2ex}%-1.1743pt}
\title[Integrality gaps]{Integrality gaps of integer knapsack problems}

\author{Iskander Aliev}
\address{Mathematics Institute, Cardiff University, UK}
\email{alievi@cardiff.ac.uk}
\author{Martin Henk}
\address{Department of Mathematics, TU Berlin, Germany}
\email{henk@math.tu-berlin.de}
\author{Timm Oertel}
\address{Mathematics Institute, Cardiff University, UK}
\email{oertelt@cardiff.ac.uk}

\date{\today}

\begin{abstract}
We obtain optimal lower and upper bounds for the (additive) integrality gaps of integer knapsack problems.
In a randomised setting, we show that the integrality gap of  a ``typical'' knapsack problem is drastically smaller
than the integrality gap that occurs in a worst case scenario.

\end{abstract}
\maketitle

%%%%%%%%%%%%%%%%%%%%%%%%%%%%%%%%%%%%%%%%%%%%%%%%%%%%%%%%%%%%%%%%%%%%%%%%%
\section{Introduction}\label{intro}%%%%%%%%%%%%%%%%%%%%%%%%%%%%%%%%%%%%
%raggedbottom%%%%%%%%%%%%%%%%%%%%%%%%%%%%%%%%%%%%%%%%%%%%%%%%%%%%%%%%%%%%

Given an integer $m\times n$ matrix $A$,
integer vector ${\ve b}\in \Z^m$ and a cost vector ${\ve c}\in \Q^n$, consider the linear integer programming problem
\be
\min\{ {\ve c} \cdot {\ve x}: A{\ve x}={\ve b}, {\ve x}\in \Z^n_{\ge 0}\}\,.
\label{initial_IP}
\ee
The linear programming relaxation to (\ref{initial_IP}) is obtained by dropping the integrality constraint
\be
\min\{ {\ve c} \cdot {\ve x}: A{\ve x}={\ve b}, {\ve x}\in \R^n_{\ge 0}\}\,.
\label{initial_LP}
\ee
We will denote by $IP_{\ve c}(A, {\ve b})$ and $LP_{\ve c}(A, {\ve b})$ the optimal values of (\ref{initial_IP}) and (\ref{initial_LP}), respectively.

While the problem \eqref{initial_LP} is polynomial time solvable~\cite{khachiyan1980polynomial}, it is well known that \eqref{initial_IP} is NP-hard~\cite{GareyJohnson-Book79}.
There are many examples, where relaxation on the integrality constraints are used to approximate, or even to solve, integer programming problems.
Prominent examples can be found in the areas of cutting plane algorithms, such us Gomory cuts \cite{Gomory58},
%, and in
and approximation algorithms for combinatorial problems. For further details see \cite{Bertsimas2005}, \cite{ConCorZam} and \cite{Vazirani}.
%combinatorial optimisation, such as approximation algorithms for the Max-Cut problem~\cite{GoemansWilliamson95}.
Therefore, a natural question is to compare the optimal values $IP_{\ve c}$ and $LP_{\ve c}$ with each other.

Suppose that (\ref{initial_IP}) is feasible and bounded.
The {\em (additive) integrality gap } $IG_{\ve c}(A, {\ve b})$ is  a fundamental characteristic of the problem (\ref{initial_IP}), defined as
%the difference
%
\bea
IG_{\ve c}(A, {\ve b})= IP_{\ve c}(A, {\ve b})-LP_{\ve c}(A, {\ve b})\,.
\eea

%This paper will focus on the parametric bounds for the additive integrality gaps.

The problem of computing upper bounds for the additive integrality gaps
has been studied by Ho\c{s}ten and Sturmfels \cite{HS}, Eisenbrand and Shmonin \cite{ES} and, more recently, by  Eisenbrand et al \cite{EHPS}.
Specifically, given a tuple $(A, {\ve c})$ one asks for the upper bounds on $IG_{\ve c}(A, {\ve b})$ as ${\ve b}$ varies. In this setting, the optimal bound is given by the {\em integer programming gap} $\gap_{\ve c}(A)$, defined by Ho\c{s}ten and Sturmfels \cite{HS} as
\bea
\gap_{\ve c}(A)= \max_{{\ve b}} IG_{\ve c}(A, {\ve b})\,,
\eea
where ${\ve b}$ ranges over integer vectors such that (\ref{initial_IP}) is feasible and bounded.
Note that, $\gap_c(A)=0$ for all $c\in\Z^n$, if and only if $A$ is totally unimodular \cite[Theorem 19.2]{Schrijver}.
Ho\c{s}ten and Sturmfels \cite{HS} showed that for fixed $n$ the value of $\gap_{\ve c}(A)$ can be computed in polynomial time.   Eisenbrand and Shmonin \cite{ES} extended this result to integer programs in the canonical form.

Eisenbrand et al \cite{EHPS} studied a closely related problem of testing upper bounds for $IG_{\ve c}(A, {\ve b})$ in context of a generalised {\em integer rounding property}. Following \cite{EHPS}, the tuple $(A, {\ve c})$ with ${\ve c}\in \Z^n$ has the {\em additive integrality gap of at most} $\gamma$ if
\bea\label{integer_rounding}
IP_{\ve c}(A, {\ve b})\le \lceil LP_{\ve c}(A, {\ve b}) \rceil +\gamma\,
\eea
for each ${\ve b}$ for which the linear programming relaxation (\ref{initial_LP}) is feasible.

The classical case $\gamma=0$ corresponds to the integer rounding property and can be tested in polynomial time \cite[Section 22.10]{Schrijver}. The integer rounding property, in its turn, implies solvability of (\ref{initial_IP}) in polynomial time \cite{Chandra}. The computational complexity of the problem drastically changes already for $\gamma=1$.
Eisenbrand et al \cite{EHPS} showed that it is NP-hard to test whether $(A, {\ve c})$ has additive gap of at most $\gamma$ even if $m=\gamma=1$.

A bound for the additive integrality gap in terms of $A$ and ${\ve c}$ can be derived from the results of Cook et al \cite{Cook} on distances between optimal solutions to integer programs in canonical form and their linear programming relaxations. Let $\hat \A$ be an integer $d\times n$ matrix and let $\hat {\ve b}$ and ${\ve c}$
be rational vectors such that $\hat A{\ve x}\le \hat {\ve b}$ has an integer solution and  $\min\{ {\ve c} \cdot {\ve x}: \hat A{\ve x}\le \hat {\ve b}, {\ve x}\in \R^{n}\}$ exists. Note that, in this setting $\hat{\ve b}$ is not required to be integer. Then Corollary 2 in \cite{Cook}, applied in the minimisation setting, gives the bound
\be\label{Cook_bound}\begin{split}
\min\{ {\ve c} \cdot {\ve x}: \hat A{\ve x}\le \hat {\ve b}, {\ve x}\in \Z^{n}\} - \min\{ {\ve c} \cdot {\ve x}: \hat A{\ve x}\le \hat {\ve b}, {\ve x}\in \R^{n}\}\\
\le
n \Delta(A) \|{\ve c}\|_1\,,
\end{split}
\ee
where $\Delta(A)$ stands for the maximum sub-determinant of $A$ and $\|{\ve c}\|_1=\sum_{i=1}^n|c_i|$ denotes the $l_1$-{\em norm} of ${\ve c}$. The estimate (\ref{Cook_bound}) strengthened previous results of Blair and Jeroslow \cite{BlairJeroslow1}, \cite{BlairJeroslow2}.
Given that $\hat {\ve b}$ does not have to be integer, one can show that the bound \eqref{Cook_bound} is essentially tight (see Remark~\ref{remark:optimal}). However, considering that we study linear integer programming, it is natural to assume that also $\hat {\ve b}$ is integer, but then it is not clear whether \eqref{Cook_bound} remains optimal. 
By studying linear integer programming problems in standard form we naturally require $\ve b$ and respectively $\hat {\ve b}$ to be integer. 
%In our result we require this implicitly.

This paper will focus on the problem (\ref{initial_IP}) in the case $m=1$, referred in the literature as the {\em integer knapsack problem}.
We will assume  that the entries of $\A$ are positive. For the integer knapsack problem the positivity assumption  guarantees that the feasible region of its linear programming relaxation (\ref{initial_LP}) is bounded (or empty) for all ${\ve b}$. Conversely, for $m=1$ any linear problem (\ref{initial_LP}) with bounded feasible region can be written with $A$ satisfying the positivity assumption.
Without loss of generality, we also assume that $n\ge 2$ and the entries of $A$ are coprime. That is the following conditions are assumed to hold:
\be \label{nonzeroA}
\begin{array}{ll}
(i) & \A=(a_1, \ldots, a_n)\,,n\ge 2\,, a_i\in \Z_{>0}\,, i=1,\ldots,n\,,\\
(ii) & \gcd(a_1, \ldots, a_n)=1\,.
\end{array}
\ee
%

%Then $A$ is an $n$-dimensional row vector with integer entries and $b$ is an integer number.
%
For $\A\in \Z^{1\times n}$ we denote by $\|\A\|_{\infty}$ its {\em maximum norm}, i.e., $ \|\A\|_{\infty} =\max_{i=1,\ldots,n} |a_i|$.
Applying (\ref{Cook_bound}) with
\bea \begin{array}{cc}
\hat \A =
\left(
\begin{array}{r}
A \\
-A  \\
-I_n
\end{array}
\right)\,,
&
\hat{\ve b}=
\left(
\begin{array}{r}
b \\
-b \\
{\ve 0}\\
\end{array}
\right)\,,
\\
\end{array}
\eea
where ${I}_n$ is the $n\times n$ identity matrix and ${\ve 0}$ is the $n$ dimensional zero vector,
we obtain the bound
\be\label{Cook_implies}
\gap_{\ve c}(\A)\le  n \|\A\|_{\infty}\|{\ve c}\|_1\,.
\ee

How far is the bound (\ref{Cook_implies}) from being optimal?
Does $\gap_{\ve c}(\A)$ admit a natural lower bound?
To answer these questions we will establish a link between the integer programming gaps, covering radii of simplices and Frobenius numbers.
Our first result gives an upper bound on the integer programming gap that improves  (\ref{Cook_implies}) with factor $1/n$. We also show that the obtained bound is optimal.

\begin{thm}\label{upper_bound}\hfill \begin{itemize}

\item[(i)] Let $\A$ satisfy (\ref{nonzeroA}) and let ${\ve c}\in \Q^n$. Then
\be\label{thm_upper}
\gap_{\ve c}(\A)\le \left( \|\A\|_{\infty}-1 \right)\|{\ve c}\|_1\,.
\ee
\item[(ii)] For any positive integer $k$ there exist $\A$ with $\|\A\|_{\infty}=k$ satisfying (\ref{nonzeroA}) and ${\ve c}\in \Q^n$ such that
\be\label{upper_bound_optimal}
\gap_{\ve c}(\A)= \left( \|\A\|_{\infty}-1 \right)\|{\ve c}\|_1\,.
\ee

\end{itemize}
\end{thm}

We will say that the tuple $(A, {\ve c})$ is {\em generic} if for any positive $b \in \Z$
the linear programming relaxation (\ref{initial_LP}) has a unique optimal solution.
An optimal lower bound for $\gap_{\ve c}(\A)$ with generic $(A, {\ve c})$ can be obtained
using recent results \cite{lpg}  on the {\em lattice programming gaps} associated with the group relaxations to (\ref{initial_IP}).

A subset $\tau$ of $\{1,\ldots,n\}$ partitions ${\ve x}\in \R^n$ as ${\ve x}_{\tau}$ and ${\ve x}_{\bar \tau}$, where ${\ve x}_{\tau}$ consists of the entries indexed by $\tau$
and ${\ve x}_{\bar \tau}$ the entries indexed by the complimentary set ${\bar \tau}=\{1,\ldots,n\}\setminus\tau$. Similarly, the matrix $A$ is partitioned as $\A_\tau$ and $\A_{\bar \tau}$.
Assume that $(A,c)$ is generic and \eqref{nonzeroA} holds.
Then, let $\tau=\tau(\A, {\ve c})$ denote the unique index of the basic variable for the optimal solution to the linear relaxation (\ref{initial_LP}) with a positive $b\in \Z$.
The index $\tau$ is well-defined. %due to the assumption (i) in (\ref{nonzeroA}).
We also define ${\ve l}(A, {\ve c})={\ve c}_{\bar \tau}-{\ve c}_{\tau}\A_{\tau}^{-1}\A_{\bar \tau}$. Note that the vector ${\ve l}={\ve l}(A, {\ve c})$  is  positive for generic tuples $(A, {\ve c})$.

%We will set, for convenience,
%
%\bea
%d=n-1\,.
%\eea
%

%
%and for ${\ve x}\in \R^d$ write
%
%\bea
%\Pi({\ve x})=x_1\cdots x_d\,.
%\eea
Let $\rho_{d}$ denote the {\em covering constant} of the standard $d$-dimensional simplex, defined in Section \ref{geometry}.

\begin{thm}\hfill
\begin{itemize}
\item[(i)] Let $\A$ satisfy (\ref{nonzeroA}) and let ${\ve c}\in \Q^n$. Suppose that $(A, {\ve c})$ is generic. Then for $\tau=\tau(\A, {\ve c})$ and ${\ve l}={\ve l}(A, {\ve c})$ we have
\be\begin{split}
\gap_{\ve c}(\A)\ge \rho_{{n-1}} (|A_{\tau}|l_1\cdots l_{n-1})^{1/{(n-1)}}-\|{\ve l}\|_1\,.
\label{optimal_bound}
\end{split}
\ee
\item[(ii)] For any $\epsilon>0$, there exists a matrix $\A$, satisfying (\ref{nonzeroA}) and ${\ve c}\in \Q^n$ such that $(A, {\ve c})$ is generic and, in the notation of part (i), we have
\be
\gap_{\ve c}(\A)< (\rho_{{n-1}}+\epsilon) (|A_\tau|l_1\cdots l_{n-1})^{1/{(n-1)}}-\|{\ve l}\|_1\,.
\label{optimality}
\ee
\end{itemize}
\label{thm_optimal_bound}
\end{thm}
The only known values of $\rho_{d}$ are $\rho_1=1$ and $\rho_2=\sqrt{3}$ (see  \cite{Fary}).
It was proved in \cite{AG}, that $\rho_{d}>(d!)^{1/d}>d/\mathrm{e}$. For sufficiently large $d$ this bound is not far from being optimal. Indeed, $\rho_{d}\le (d!)^{1/d}(1+O(d^{-1}\log d))$ (see \cite{DF} and \cite{MS}).

%\gap({\ve a}, {\ve c})\le \left(\frac{2}{m} \|{\ve a}\|_{\infty}-1 \right) \|{\ve c}\|_1\,.

How large is the integer programming gap of a ``typical'' knapsack problem? To tackle this question we will utilize the recent strong results of Str\"ombergsson \cite{Str} (see also Schmidt \cite{Schmidt} and references therein) on the asymptotic distribution of Frobenius numbers.
The main result of this paper will show that for any $\epsilon>2/n$
the ratio
\bea
\frac{\gap_{\ve c}(\A)}{\|\A\|_{\infty}^{\epsilon}\|{\ve c}\|_1}
\eea
is bounded, on average, by a constant that depends only on dimension $n$.
Hence, for fixed $n>2$ and a ``typical'' integer knapsack problem with large $\|\A\|_{\infty}$,
its linear programming relaxation provides a drastically better approximation to the solution than in the worst case scenario, determined by the optimal upper bound (\ref{thm_upper}).

For $T\ge 1$, let ${\NC}(T)$ be the set of $\A\in \Z^{1\times n}$ that satisfy (\ref{nonzeroA}) and
\bea \|A\|_{\infty}\le T\,.\eea
Let
$N(T)$ be the cardinality of ${\NC}(T)$. For $\epsilon\in (0,1)$ let
\be
 N_{\epsilon}(t,T)=\#\left\{\A\in {\NC}(T): \max_{{\ve c}\in \Q^n}\frac{\gap_{\ve c}(\A)}{\|\A\|_{\infty}^{\epsilon}\|{\ve c}\|_1}>t \right\}\,.
\ee

In what follows, $\ll_n$ will denote the Vinogradov symbol with the constant depending on $n$. That is  $f\ll_n g$ if and only if $|f|\le c|g|$, for some positive constant $c=c(n)$.
The notation $f\asymp_n g$ means that both $f\ll_n g$ and $g\ll_n f$ hold.

%The main contribution of this paper is the following theorem.
%Using results of Str\"ombergsson \cite{Str} we obtain

\begin{thm}\label{Ratio} For $n\ge 3$
\be\label{main_bound}
\frac{N_{\epsilon}(t,T)}{N(T)}\ll_n t^{-\alpha(\epsilon,n)}
\ee
uniformly over all $t>0$ and $T\ge 1$. Here
\bea
\alpha(\epsilon,n)=\frac{n-2}{(1-\epsilon)n}\,.
\eea
\end{thm}

From (\ref{main_bound}) one can derive an upper  bound on the average
value of the (normalised) integer programming gap.

\begin{cor} \label{average} Let $n\geq 3$. For $\epsilon>2/n$ 
\be\label{average_ineq}
\frac{1}{N(T)}\sum_{\A\in \NC(T)}\max_{{\ve c}\in \Q^n}\frac{\gap_{\ve c}(\A)}{\|\A\|_{\infty}^{\epsilon}\|{\ve c}\|_1} \ll_n 1\,.
\ee
\end{cor}

The last theorem of this paper shows that the bound in Corollary \ref{average} is not far from being optimal. 
We include its proof in the Appendix.
\begin{thm}\label{lower_upper}
For $T$ large
\be\label{average_ineq_lo}
\frac{1}{N(T)}\sum_{\A\in \NC(T)}\max_{{\ve c}\in \Q^n}\frac{\gap_{\ve c}(\A)}{\|\A\|_{\infty}^{1/(n-1)}\|{\ve c}\|_1} \gg_n 1\,.
\ee
\end{thm}
Hence, the optimal value of $\epsilon$ in (\ref{average_ineq})
cannot be smaller than $1/(n-1)$.

\begin{remark}\label{remark:optimal} \hfill \begin{itemize}
\item[(i)]
An example due to L. Lov\'asz \cite[Section 17.2]{Schrijver}, with $\Delta(A)=1$, shows that the bound \eqref{Cook_bound} is best possible in this particular case. We would like to point out that by a small adaptation of  Lov\'asz's example one can show that this bound is, in all its generality, best possible up to a constant factor, i.e., the upper bound for the additive integrality gap is in $\Theta(\Delta(A) n)$.
Let $\delta\in\Z_{>0}$ and $0<\beta<1$.
We define
\begin{align*}
%\begin{small}
A=\left( \begin{matrix}
1 &  \\
-1& 1 &  \\
 & & \ddots \\
 & & -1 & 1 \\
 & & & -\delta & 1
\end{matrix} \right),\;
\ve b=\left( \begin{matrix}
\beta  \\
\vdots  \\
\beta
\end{matrix} \right)
\text{ and }
\ve c=\left( \begin{matrix}
-1 \\ \vdots \\ -1
\end{matrix} \right).
%\end{small}
\end{align*}
By construction $\Delta(A)=\delta$. The unique solution of the linear relaxation is $\ve x^T=(\beta,2\beta,\ldots,(n-1)\beta,(\delta(n-1)+1)\beta)$ and the unique optimal integer solution is $\ve z^T=(0,\ldots,0)$.
Thus $\|\ve x - \ve z\|_\infty = (\delta(n-1)+1)\beta \approx n \Delta(A)$.
\item[(ii)]
In the proof of Theorem \ref{upper_bound} (and, subsequently, Theorem \ref{Ratio}) we estimate the integrality gap using a covering argument that guarantees existence of a solution to (\ref{initial_IP}) in an $(n-1)$-dimensional simplex of sufficiently small diameter, translated by a solution to (\ref{initial_LP}).
Here the diameter of the simplex is independent of ${\ve c}$. The argument allows us, in particular, to
restate Theorem~\ref{upper_bound} (i) in terms of the infinity norm:
\bea\gap_{\ve c}(\A)\le 2 \left( \|\A\|_{\infty}-1 \right)\|{\ve c}\|_\infty\,.\eea
Depending on $\ve c$ this gives a stronger bound.
\end{itemize}
\end{remark}

\section{Coverings and Frobenius numbers}\label{geometry}

In what follows, $\mathcal{K}^d$ will denote the space of all $d$-dimensional {\em convex bodies}, i.e., closed bounded convex sets with non-empty interior in
the $d$-dimensional Euclidean space $\R^d$.
%For $K\in \mathcal{K}^d$, the {\em difference body} $D_K$ of $K$ is the origin-symmetric convex body defined as $D_K=K-K=K+(-K)$.

%The volume of a set $X\subset\R^d$, i.e., its $d$-dimensional Lebesgue measure, is denoted by $\vol(X)$.
By $\mathcal{L}^d$ we denote the set of all $d$-dimensional lattices in $\R^d$. Given a matrix $B\in\R^{d\times d}$ with $\det B\neq 0$ and a set $S \subset \R^d$
let $BS=\{B{\ve x}: {\ve x}\in S\}$ be the image of $S$ under linear map defined by $B$. Then we can write $\mathcal{L}^d=\{B\,\Z^d : B\in\R^{d\times d},\,\det B\ne 0\}$.
For $\Lambda=B\,\Z^d\in\mathcal{L}^d$, $\det(\Lambda)=|\det B|$ is
called the {\em determinant} of the lattice $\Lambda$.

Recall that the {\em Minkowski sum} $X+Y$ of the sets $X, Y\subset \R^d$ consists of all points ${\ve x}+{\ve y}$ with ${\ve x}\in X$ and ${\ve y}\in Y$.
For $K\in\mathcal{K}^d$ and $\Lambda\in\mathcal{L}^d$ the {\em covering radius} of $K$ with respect to $\Lambda$ is
the smallest positive number $\mu$ such that any point ${\ve x}\in\R^d$ is  covered
by $\mu\, K+ \Lambda$, that is
\begin{equation*}
 \mu(K,\Lambda)  =\min\{\mu > 0 : \R^d = \mu K+ \Lambda\}\,.
\end{equation*}
For further information on covering radii in the context of the geometry of numbers see e.g. Gruber \cite{peterbible} and Gruber and Lekkerkerker \cite{GrLek}.

Let $\Delta=\{{\ve x}\in \R^d_{\ge 0}: x_1+\cdots+x_d\le 1\}$ be the standard $d$-dimensional simplex.
The optimal lower bound in Theorem \ref{thm_optimal_bound} is expressed using the covering constant
$\rho_d=\rho_d(\Delta)$ defined as
\bea
\rho_d=\inf\{\mu(\Delta, \Lambda): \det(\Lambda)=1 \}\,.
\eea

We will be also interested in coverings of $\Z^d$ by lattice translates of convex bodies. For this purpose we define
\begin{equation*}
\begin{split}
 \mu(K,\Lambda; \Z^d)  =\min\{\mu > 0 : \Z^d \subset \mu K+ \Lambda\}\,.
\end{split}
\end{equation*}

Given $\A=(a_1, \ldots, a_n)$ satisfying (\ref{nonzeroA})
the {\em Frobenius number} $\frob(\A)$ is least so that every integer $b>\frob(\A)$ can be represented as
$
b= a_1 x_1 +\cdots+ a_n x_n
$
with nonnegative integers $x_1,\ldots, x_n$.

Kannan \cite{Kannan} found a nice and very useful connection between $\frob(\A)$ and geometry of numbers.
Let us consider the $(n-1)$-dimensional simplex
\begin{equation*}
S_\A = \left\{ {\ve x} \in \R_{\geq 0}^{n-1} : a_1\,x_1+\cdots +a_{n-1}\,x_{n-1}\leq 1 \right\}
\end{equation*}
and the $(n-1)$-dimensional lattice
\begin{equation*}
\Lambda_\A = \left\{ {\ve x}\in\Z^{n-1} : a_1\,x_1+\cdots + a_{n-1}\,x_{n-1}\equiv 0 \bmod a_n \right\}.
\end{equation*}
Kannan \cite{Kannan} established the  identities
\begin{equation*}
       \mu(S_\A,\Lambda_\A)= \frob(\A)+a_1+\cdots +a_n
\end{equation*}
and
\begin{equation}\label{integralKannan}
       \mu(S_\A,\Lambda_\A; \Z^{n-1})= \frob(\A)+a_n.
\end{equation}

\section{Proof of Theorem \ref{upper_bound}}
\label{section_upper_bound}

The proof of the upper bound in part (i) will be based on two auxiliary lemmas.
First we will need the following property of  $\mu(K,\Lambda; \Z^{n-1})$.
\begin{lemma}\label{lemma_lattice_covering}
For any ${\ve y}\in \Z^{n-1}$ the set $\mu(K,\Lambda; \Z^{n-1}) K$ contains a point of the translated lattice ${\ve y} + \Lambda$.
\end{lemma}
\begin{proof}
By the definition of $\mu(K,\Lambda; \Z^{n-1})$ we have  $\Z^{n-1} \subset \mu(K,\Lambda; \Z^{n-1}) K + \Lambda$. Therefore
for any integer vector ${\ve y}$ we have
$({\ve y}+\Lambda)\cap \mu(K,\Lambda; \Z^{n-1}) K\neq \emptyset$.
\end{proof}

The next lemma gives an upper bound for the integer programming gap in terms of the Frobenius number associated with vector $\A$.
\begin{lemma}\label{lemma_bounds_for_distance} For $\A$ satisfying (\ref{nonzeroA}) and ${\ve c}\in \Q^n$
\be\label{bound_via_Frobi}
\gap_{\ve c}(\A)\le \frac{(\frob(\A)+\|\A\|_\infty)\|{\ve c}\|_1}{\min_i a_i}\,.
\ee
\end{lemma}

\begin{proof}

Let $b$ be a nonnegative integer. Consider the {\em knapsack polytope}
\bea
P(\A, b)=\{{\ve x}\in \R^n_{\ge 0}: \A {\ve x}=b\}\,.
\eea
Clearly, $P(\A, b)$ is a simplex with vertices
\bea
(b/a_1,0, \ldots, 0), (0, b/a_2, \ldots, 0), \ldots, (0, \ldots, 0, b/a_n)
\eea
and
\be\label{cube}
P(\A, b)\subset \left[0, \frac{b}{\min_i a_i}\right]^n\,.
\ee

Notice also that
\be\label{projection}
bS_\A= \pi_n(P(\A, b))\,,
\ee
where $\pi_n(\cdot): \R^n\rightarrow \R^{n-1}$ is the projection that forgets the last coordinate.

Rearranging the entries of $\A$, if necessary, we may assume that the optimal value $LP_{\ve c}(\A,b)$ is attained at the vertex
${\ve v}=(0,\ldots, 0, b/a_n)$ of $P(\A, b)$.

If $b \le \mu(S_\A,\Lambda_\A; \Z^{n-1})$ then (\ref{integralKannan}) and (\ref{cube}) imply that the integrality gap is bounded by the right hand side of (\ref{bound_via_Frobi}).

Suppose now that $b > \mu(S_\A,\Lambda_\A; \Z^{n-1})$. Then, in view of (\ref{projection}),
\be
\mu(S_\A,\Lambda_\A; \Z^{n-1}) S_\A\subset \pi_n(P(\A, b))\,.
\ee
Let $\Lambda(\A,b)=\{{\ve x}\in \Z^n: \A{\ve x}=b\}$ be the set of integer points in the affine hyperplane $\A{\ve x}=b$.
There exists ${\ve y}\in \Z^{n-1}$ such that
\be
\pi_n(\Lambda(\A,b))={\ve y}+\Lambda_\A\,.
\ee
By Lemma~\ref{lemma_lattice_covering},
there is a point $(z_1, \ldots, z_{n-1})\in \pi_n(\Lambda(\A,b))\cap \mu(S_\A,\Lambda_\A; \Z^{n-1}) S_\A$.
Hence
\be
{\ve z}=\left(z_1, \ldots, z_{n-1}, \frac{b}{a_n}-\frac{a_1z_1+\cdots+a_{n-1}z_{n-1}}{a_n}\right)\in \Lambda(\A,b)\cap P(\A, b)
\ee
is a feasible integer point for the knapsack problem (\ref{initial_IP}).

Since $(z_1, \ldots, z_{n-1})\in \mu(S_\A,\Lambda_\A; \Z^{n-1}) S_\A$, we have
\be
||{\ve v}-{\ve z}||_\infty \le \frac{\mu(S_\A,\Lambda_\A; \Z^{n-1})}{\min_i a_i}\le \frac{\frob(\A)+\|\A\|_\infty}{\min_i a_i}\,,
\ee
where the last inequality follows from (\ref{integralKannan}).
Therefore, the integrality gap is bounded by the right hand side of (\ref{bound_via_Frobi}).
\end{proof}

To complete the proof of part (i) we need the classical upper bound for the Frobenius number due to Schur (see Brauer \cite{Brauer}):
\be\label{Schur}
\frob(\A)\le (\min_i a_i)\|\A\|_{\infty}-(\min_i a_i)-\|\A\|_{\infty}\,.
\ee
Combining (\ref{bound_via_Frobi}) and  (\ref{Schur}) we obtain (\ref{thm_upper}).

To prove part (ii), we set $A=(k,\ldots,k,1)$, $b=k-1$ and $\ve{c}=\ve{e}_n$, where $\ve{e}_i$ denotes the $i$-th unit-vector.
Note that $A$ fulfils the conditions~~\eqref{nonzeroA}.
The integer programming problem~\eqref{initial_IP} has precisely one feasible, and therefore optimal, integer point, namely $(k-1)\cdot\ve{e}_n$. Thus $IP_{\ve{c}}(A,b)=k-1$.
The corresponding linear relaxation~\eqref{initial_LP} has the, in general not unique, optimal solution $\frac{k-1}{k}\cdot\ve{e}_1$ with $LP_{\ve{c}}(A,b)=0$.
Hence, $\gap_{\ve{c}}(A) \ge IG_{\ve{c}}(A,b) = k-1 = (\|\A\|_{\infty}-1) \|{\ve c}\|_1$.
%Together with Theorem~\ref{upper_bound}, the statement follows.
%{\bf REMARK:} Depending on the norms, one can possibly improv the bound slightly. For example, by using $\ve{c}:=(2,0,\ldots,0,a-1)$ ...
\qed

%\begin{lemma}\label{lemma_bounds_for_distance} For vectors $\A$ with $a_1\le\cdots \le a_m$
%\be
%\frac{\frob(\A)\|{\ve c}\|_1}{a_1}\le \gap(\A, {\ve c})\le \frac{(\frob(\A)+a_m)\|{\ve c}\|_1}{a_1}\,.
%\ee
%\end{lemma}

\section{Proof of Theorem \ref{thm_optimal_bound}}

We will first establish a connection between $\gap_{\ve c}(\A)$ and the lattice programming gap
associated with a certain lattice program.

For a  vector ${\ve w}\in \Q^{n-1}_{>0}$, a $(n-1)$-dimensional lattice $\Lambda\subset\Z^{n-1}$ and ${\ve r}\in \Z^{n-1}$
consider the lattice program (also referred to as the {\em group problem})
\be\begin{split}
\min\{ {\ve w}\cdot {\ve x}: {\ve x} \equiv {\ve r} (\modulo \Lambda), {\ve x}\in \R^{n-1}_{\ge 0}\}\,.
\end{split}
\label{generic_group_relaxation}
\ee
Here ${\ve x} \equiv {\ve r} (\modulo \Lambda)$ if and only if ${\ve x} - {\ve r}$ is a point of $\Lambda$.

Let $m(\Lambda,{\ve w}, {\ve r})$ denote the value of the minimum in (\ref{generic_group_relaxation}).
%One can view  $m_{l}(\Lambda,r)$ as the {\em integer programming gap} of the program (\ref{generic_group_relaxation}).
The {\em lattice programming gap} $\gap(\Lambda, {\ve w})$ of (\ref{generic_group_relaxation}) is defined as
\be\begin{split}
\gap(\Lambda,{\ve w})=\max_{{\ve r}\in \Z^{n-1}}m(\Lambda,{\ve w}, {\ve r})\,.
\end{split}
\label{maximum_gap}
\ee
The lattice programming gaps were introduced and studied for sublattices of all dimensions in $\Z^{n-1}$ by Ho\c{s}ten and Sturmfels \cite{HS}.

To proceed with the proof of the part (i), we assume without loss of generality that $\tau(\A, {\ve c})=\{n\}$.
Then for ${\ve l}={\ve l}({\A},{\ve c})$ the lattice programs
\be\label{technical}\begin{split}
\min\{ {\ve l}\cdot {\ve x}: {\ve x} \equiv {\ve r}\; (\modulo \Lambda_{\A}), {\ve x}\in \R^{n-1}_{\ge 0}\}\,,\; {\ve r}\in \Z^{n-1}
\end{split}
\ee
are the {\em group relaxations} to (\ref{initial_IP}).

Indeed, for any positive $b\in \Z$ and any integer solution ${\ve z}$ of the equation $A{\ve x}=b$ the lattice program (\ref{technical})
with ${\ve r}=\pi_n({\ve z})$,
is a group relaxation to (\ref{initial_IP}).
On the other hand, for any integer vector ${\ve r}$ the lattice program (\ref{technical})
is a group relaxation to (\ref{initial_IP}) with
$b=\pi_n(A){\ve u}$
for a nonnegative integer vector ${\ve u}$ from ${\ve r}+\Lambda_{\A}$.

In both cases
\bea
IG_{\ve c}(\A, b)\ge m(\Lambda_{\A},{\ve l}, {\ve r})
\eea
and, consequently,
\be\label{IPGbiggerLPG}
\gap_{\ve c}(\A)\ge \gap(\Lambda_A, {\ve l})\,.
\ee
Note that for $n=2$ we have $\gap(\Lambda_A, {\ve l})=l_1(|\A_{\tau}|-1)$ and thus (\ref{IPGbiggerLPG}) implies (\ref{optimal_bound}).
For $n>2$, the bound (\ref{optimal_bound}) immediately follows from (\ref{IPGbiggerLPG}) and Theorem 1.2(i) in \cite{lpg}.

The proof of the part (ii) will be based on the following lemma.

\begin{lemma}\label{Lpg_link} Let $\A$ satisfy (\ref{nonzeroA}),   ${\ve c}=(a_1, \ldots, a_{n-1}, 0)^t\in \Q^n$ and ${\ve l}=(a_1, \ldots, a_{n-1})^t\in \Q^{n-1}_{>0}$. Then
\be\label{Lpg_via_gap}
\gap_{\ve c}(\A)= \gap(\Lambda_A, {\ve l})\,.
\ee
\end{lemma}

\begin{proof}
Observe that assumption (i) in (\ref{nonzeroA}) implies that the linear programming relaxation (\ref{initial_LP}) is feasible if and only if $b$ is nonnegative.
Recall that $\Lambda(\A,b)=\{{\ve x}\in \Z^n: \A{\ve x}=b\}$ denotes the set of integer points in the affine hyperplane $\A{\ve x}=b$ and $P(\A,b)=\{{\ve x}\in\R_{\ge 0}: \A{\ve x}=b\}$ denotes the knapsack polytope.
Suppose that for a nonnegative $b$ the knapsack problem (\ref{initial_IP}) is feasible, with solution ${\ve y}\in \Z^n_{\ge 0}$. Then  for ${\ve r}=\pi_n({\ve y})\in \Z^{n-1}_{\ge 0}$
\bea
\pi_n(\Lambda(\A,b))={\ve r}+\Lambda_\A\,.
\eea
As $c_n=0$, the optimal value of the linear programming relaxation $LP_{{\ve c}}(A, b)=0$.
Therefore, noting that ${\ve c}=(a_1, \ldots, a_{n-1}, 0)^t$ and ${\ve l}=\pi_n({\ve c})$,
\be\label{Extra_constraint}
IG_{\ve c}(A, b)=\min\{{\ve l}\cdot {\ve x}: {\ve x}\in {\ve r}+\Lambda_\A\,, {\ve x}\in \pi_n(P({\A}, b)) \}\,.
\ee
Since
\bea
\pi_n(P({\A}, b))=bS_{\A}=\{{\ve x}\in \R^{n-1}_{\ge 0}: {\ve l}\cdot {\ve x}\le b\}\,
\eea
and ${\ve l}\cdot {\ve r}\le A{\ve y}=b$, the constraint ${\ve x}\in \pi_n(P({\A}, b))$ in (\ref{Extra_constraint}) can be removed. Consequently, we have
\bea %\label{Gaps_are_equal}
IG_{\ve c}(A, b)= m(\Lambda_A,{\ve l}, {\ve r})\,.
\eea
%Notice that (\ref{Gaps_are_equal}) holds due to a special choice of ${\ve c}$ and ${\ve l}$.
Hence, by (\ref{maximum_gap}), we obtain
\be\label{gap_less}
\gap_{\ve c}(\A)\le  \gap(\Lambda_A, {\ve l})\,.
\ee
Suppose now that $\gap(\Lambda_A, {\ve l})= m(\Lambda_A,{\ve l}, {\ve r}_0)$. Then
\bea
IG_{\ve c}(\A, \A{\ve r}_0)= m(\Lambda,{\ve l}, {\ve r}_0)\,.
\eea
 Together with (\ref{gap_less}), this implies (\ref{Lpg_via_gap}).
\end{proof}

As was shown in the proof of Theorem 1.1 in \cite{lpg}, for ${\ve l}=(a_1, \ldots, a_{n-1})^t$
\bea
\gap(\Lambda_A, {\ve l})=\frob(\A)+a_n\,.
\eea
Thus we obtain the following corollary.

\begin{cor}\label{Frobi_link} Let $\A=(a_1, \ldots, a_n)$ satisfy (\ref{nonzeroA}) and ${\ve c}=(a_1, \ldots, a_{n-1}, 0)^t$. Then
\be\label{Frobi_via_gap}
\gap_{\ve c}(\A)= \frob(\A)+a_n\,.
\ee
\end{cor}

%Recall that by Corollary \ref{Frobi_link} for any $\A=(a_1, \ldots, a_n)$ satisfying (\ref{nonzeroA}) and ${\ve c}=(a_1, \ldots, a_d, 0)^t$
%\be\label{Frobi_link2}
%\gap_{\ve c}(\A)= \frob(\A)+a_n\,.
%\ee
For $n=2$, we have
\be\label{Sylvester}
g(A)=a_1a_2-a_1-a_2
\ee
 by a classical result of Sylvester (see e.g. \cite{Alf}). Hence
the part (ii) immediately follows from Corollary \ref{Frobi_link}.
For $n>2$, noting that $|\A_{\tau}|=a_n$, %\det(\Lambda_A)=a_n$,
the part (ii) follows from  Corollary \ref{Frobi_link} and Theorem 1.1 (ii) in \cite{AG}.

\section{Proof of Theorem \ref{Ratio}}

For convenience, we will work with the quantity
\bea
\ffrob(\A)= \frob(\A)+a_1 +\cdots+ a_n
\eea
and the set
\bea
\RR= \{A\in \Z^{1\times n}: 0<a_1\le\cdots \le a_n\} \,.
\eea
%
%and also assume without loss of generality
%\be\label{nondecreasing}
%%a_1\le\cdots \le a_n\,.
%\ee
%
%Notice that for $A\in \RR$ we have $\min_i a_i=a_1$ and $\|A\|_{\infty}=a_n$.
By Lemma \ref{lemma_bounds_for_distance}, we have
%
%Let $Q(\A)= \|\A\|_{\infty}^{1/n+\nu}=a_m^{1/n+\nu}$.
%
\be\label{basic_bound}
\begin{split}
 N_{\epsilon}(t,T)\le n!\, \#\left\{\A\in {\NC}(T)\cap \RR: \frac{\ffrob(\A)}{a_1 a_n^{\epsilon}}>t \right\} \,.
\end{split}
\ee

We may assume $t \ge 10$ since otherwise (\ref{main_bound})
follows from $N_{\epsilon}(t,T)/N(T)\le 1$.
%For $n=2$, the inequality (\ref{basic_bound}) together with (\ref{Sylvester})
%implies $N_{\epsilon}(t,T)=0$ for $t>3$.
%Therefore, we may also assume without loss of generality that $n>2$.
%Next
%
%\be
%\begin{split}
%\#\left\{\A\in {\NC}(T): \frac{\ffrob(\A)}{a_1 Q(\A)}>t \right\} \\
%= \#\left\{\A\in {\NC}(T): \frac{\ffrob(\A)}{S(\A)}\cdot\frac{S(\A)}{a_1Q(\A)}>t \right\}\,.
%\end{split}
%\ee
We keep $t' \in [1, t]$, to be fixed later. Then, setting $s(\A)= a_{n-1} a_n^{1/(n-1)}$ and noting (\ref{basic_bound}), we get
\be\label{two_terms}
\begin{split}
 %N_{\epsilon}(t,T)\ll_n \#\left\{\A\in {\NC}(T)\cap \RR: \frac{\ffrob(\A)}{a_1 a_n^{1/(n-1)+\nu}}>t \right\} \\
 N_{\epsilon}(t,T)\le n!\, \#\left\{\A\in {\NC}(T)\cap \RR: \frac{\ffrob(\A)}{s(\A)}>t'  \mbox{ or } \frac{s(\A)}{a_1a_n^{\epsilon}} > \frac{t}{t'}\right\}
\\
\le n!\,\#\left\{\A\in {\NC}(T)\cap \RR: \frac{\ffrob(\A)}{s(\A)}>t'\right\} \\ + \,n!\,\#\left\{\A\in {\NC}(T)\cap \RR: \frac{a_{n-1}}{a_1 a_n^{\epsilon-1/(n-1)}} > \frac{t}{t'}\right\}\,.
\end{split}
\ee

The first of the last two terms in (\ref{two_terms}) can be estimated using a special case of Theorem 3 in Str\"ombergsson \cite{Str}.

\begin{lemma}\label{first_term_bound}

\be\label{Str_bound}
\#\left\{\A\in {\NC}(T)\cap\RR: \frac{\ffrob(\A)}{s(\A)}>r\right\}\ll_n \frac{1}{r^{n-1}} N(T) \,.
\ee
\end{lemma}
\begin{proof}
The inequality (\ref{Str_bound}) immediately follows from Theorem 3 in \cite{Str} applied with
${\mathcal D}=[0,1]^{n-1}$.
\end{proof}
To estimate the last term, we will need the following lemma.

\begin{lemma}\label{second_term_bound}

\be\label{second_term_ineq}
\#\left\{\A\in {\NC}(T)\cap \RR: \frac{a_{n-1}}{a_1 a_n^{\epsilon-1/(n-1)}} > r\right\}\ll_n \frac{1}{rT^{\epsilon-1/(n-1)}}N(T) \,.
\ee
\end{lemma}

\begin{proof}

Since $A\in \RR$, we have $a_{n-1}\le a_n$. Hence
\bea
\#\left\{\A\in {\NC}(T)\cap \RR: \frac{a_{n-1}}{a_1 a_n^{\epsilon-1/(n-1)}} > r\right\}
\le  \#\left\{\A\in {\NC}(T)\cap \RR: a_n^{1+1/(n-1)-\epsilon}> r a_1\right\}\,.
\eea
Furthermore, all $\A\in {\NC}(T)\cap \RR$ with $a_n^{1+1/(n-1)-\epsilon}> r a_1$ are in the set
\bea
U=\{\A\in \Z^{1\times n}: 0< a_1< T^{1+1/(n-1)-\epsilon}/r, 0< a_i\le T, i=2, \ldots, n\}\,.
\eea
Since $\#(U\cap\Z^n)< T^{n+1/(n-1)-\epsilon}/r$ and $N(T)\asymp_n T^n$ (see e.g. Theorem 1 in \cite{SchmidtDuke}), the result follows.
\qed

%Assume now $\nu>1/2$.

%Hence
%
%\be
%\begin{split}
%\frac{\#\{\A\in {\NC}(T): \frac{a_n}{a_1 a_m^{\nu}} > r\}}{N(T)} \ll \frac{1}{r^{1/\nu}T^{2-1/\nu}} \ll
%\frac{1}{r^{1/\nu}r^{\frac{n}{n(1-\nu)-1}(2-1/\nu)}}
%\\
%=\frac{1}{r^{\frac{n\nu-1}{\nu(n(1-\nu)-1)}}}=r^{-\frac{n\nu-1}{\nu(n(1-\nu)-1)}}\,.
%\end{split}
%\ee

%\qed

Then by (\ref{two_terms}), (\ref{Str_bound}) and (\ref{second_term_ineq})

\be\label{two_terms2}
 \frac{N_{\epsilon}(t,T)}{N(T)}
\ll_n \frac{1}{(t')^{n-1}} + \frac{t'}{tT^{\epsilon-1/(n-1)}}\,.
% = \frac{1}{(t')^n}+\frac{1}{T^{2-1/\nu}} \left(\frac{t'}{t}\right)^{1/\nu}
%\\
%\ll_n \frac{1}{(t')^n}+\frac{(t')^{1/\nu}}{t^{\frac{n\nu-1}{\nu(n(1-\nu)-1)}}}=(t')^{-n}+(t')^{1/\nu}t^{-\frac{n\nu-1}{\nu(n(1-\nu)-1)}}\,.
\ee

Next, we will bound $T$ from below in terms of $t$, similar to Theorem 3 in \cite{Str}.  The upper bound of Schur (\ref{Schur})  implies
$\ffrob(\A)< n a_1 a_n$.
Thus, using (\ref{basic_bound}),
\bea
\begin{split}
 N_{\epsilon}(t,T)\le  \#\left\{\A\in {\NC}(T)\cap \RR: \frac{\ffrob(\A)}{a_1 a_n^{\epsilon}}>t \right\} \\ \le \#\left\{\A\in {\NC}(T)\cap \RR: a_n^{1-\epsilon}>\frac{t}{n} \right\}\,.
\end{split}
\eea

The latter set is empty if $T \le (t/n)^{\frac{1}{1-\epsilon}}$. Hence we may assume
\be\label{Tviat}
T > \left(\frac{t}{n}\right)^{\frac{1}{1-\epsilon}}\,.
\ee

Using (\ref{two_terms2}) and (\ref{Tviat}), we have
\be\label{viaalpha}
\frac{N_{\epsilon}(t,T)}{N(T)}\ll_n  \frac{1}{(t')^{n-1}} + \frac{t'}{  t^{ 1+\frac{ 1 }{ 1-\epsilon } \left(\epsilon -\frac{1}{n-1}\right) }  }\,.
\ee

To minimise the exponent of the right hand side of (\ref{viaalpha}), set  $t'=t^\beta$ and choose $\beta$ with
\be\label{equal_powers}
\beta (n-1)= 1+\frac{ 1 }{ 1-\epsilon } \left(\epsilon -\frac{1}{n-1}\right)-\beta\,.
\ee
We get
\bea
\beta= \frac{n -2}{n(n-1)(1-\epsilon)}
\eea
and, by (\ref{viaalpha}) and (\ref{equal_powers}),
\bea
 \frac{N_{\epsilon}(t,T)}{N(T)} \ll_n t^{-\alpha(\epsilon,n)}\,
\eea
with $\alpha(\epsilon,n)=\beta (n-1)$. The theorem is proved.
\end{proof}

\section{Proof of Corollary \ref{average}}
For the upper boudn we observe, that
 the conditions $n\ge 3$ and $\epsilon>2/n$ imply that in (\ref{main_bound}) $\alpha(\epsilon, n)>1$.
Consider vectors $\A\in \NC(T)$ with
\be\label{exponents}
e^{s-1}\le \max_{{\ve c}\in \Q^n}\frac{\gap_{\ve c}(\A)}{\|\A\|_{\infty}^{\epsilon}\|{\ve c}\|_1}<e^s\,.
\ee
The contribution of vectors satisfying (\ref{exponents}) to the sum
\bea
\sum_{\A\in \NC(T)}\max_{{\ve c}\in \Q^n}\frac{\gap_{\ve c}(\A)}{\|\A\|_{\infty}^{\epsilon}\|{\ve c}\|_1}
\eea
on the left hand side of (\ref{average_ineq}) is
\bea
\le N_{\epsilon}(e^{s-1}, T)e^{s}\ll_n e^{-\alpha(\epsilon, n) s}e^s N(T)\,,
\eea
where the last inequality holds by (\ref{main_bound}).
Therefore
\bea
\frac{1}{N(T)}\sum_{\A\in \NC(T)}\max_{{\ve c}\in \Q^n}\frac{\gap_{\ve c}(\A)}{\|\A\|_{\infty}^{\epsilon}\|{\ve c}\|_1}\ll_n \sum_{s=1}^\infty e^{s(1-\alpha(\epsilon, n))}\,.
\eea
Finally, observe that the series
\bea
\sum_{s=1}^\infty e^{s(1-\alpha(\epsilon, n))}
\eea
is convergent for $\alpha(\epsilon, n)>1$.

% \subsection*{Acknowledgements}

%%%%%%%%%%%%%%%%%%%%%%%%%%%%%%%%%%%%%%%%%%%%%%%%%%%%%%%%%%%%%%%%%%%%%%%%%
%%%%%%%%%%%%%%%%%%%%%%%%%%%%%%%%%%%%%%%%%%%%%%%%%%%%
%%%%%%%%%%%%%%%%%%%%%%%%%%%%%%%%%%%%%%%%%%%%%%%%%%%%%%%%%%%%%%%%%%%%%%%%%

\newpage

\section{Appendix: Proof of Theorem \ref{lower_upper}}

We will
denote for $\ve A\in Q(T)$ the index of a maximum coordinate by
$i(A)$ and we set ${\ve c}_{A}=-\ve e_{i(A)}$.  % Furthermore let
% \be
% Q^i(T)=\{ A \in Q(T): {\ve c}_{A}=-\ve e_i \}, \quad 1\leq i\leq n
% \ee
The tuples $(\ve A,\ve c_A )$ are generic and in view of Theorem
\ref{optimal_bound} we find
\begin{equation*}
\begin{split}
\gap_{c_A}(\ve A)  & \geq \rho_{n-1} a_{i(A)}^{1/(n-1)}\left(\prod_{i=1, i\ne
  i(A)}^{n}\frac{a_i}{a_{i(A)}}\right)^{1/(n-1)} - \sum_{i=1, i\ne i(A)}^n
\frac{a_i}{a_{i(A)}}  \\
& \geq  \frac{1}{\|\ve A\|_\infty}\rho_{n-1} \left(\prod_{i=1}^{n} a_i\right)^{1/(n-1)}-n.
\end{split}
\end{equation*}
Hence
\begin{equation*}
\frac{\gap_{c_A}(\ve A)}{\|\ve A\|_\infty^{1/(n-1)}} \geq
\rho_{n-1} \frac{1}{\|\ve A\|_\infty^{1+1/(n-1)}}
\left(\prod_{i=1}^{n} a_i\right)^{1/(n-1)}-\frac{n}{\|\ve A\|_\infty^{1/(n-1)}}.
\end{equation*}
Next we observe that
\begin{equation*}
\begin{split}
\sum_{\ve A\in Q(T)} \frac{1}{\|\ve A\|_\infty^{1/(n-1)}}& \leq
n\,T^{n-1}\sum_{t=1}^T  \frac{1}{t^{1/(n-1)}} \leq
n\,T^{n-1}\left(1+\int_{1}^T  \frac{1}{t^{1/(n-1)}}{\rm d}t\right)\\
&\leq \frac{n-1}{n-2}n\,T^{n-1} T^{1-1/(n-1)} \leq 2\,n\, T^{n-1/(n-1)}
\end{split}
\end{equation*}
for $n\geq 3$.
Thus, so far we know that
\begin{equation*}
\begin{split}
\frac{1}{N(T)}\sum_{\ve A\in Q(T)} & \max_{c\in\Q^n}\frac{\gap_{c}(\ve A)}{\|\ve
  A\|_\infty^{1/(n-1)}\|\ve c\|_1}  \geq \frac{1}{N(T)}
\sum_{\ve A\in Q(T)}
\frac{\gap_{c_A}(\ve A)}{\|\ve A\|_\infty^{1/(n-1)}}\\
&\geq
\rho_{n-1}\frac{1}{N(T)} \sum_{\ve A\in Q(T)} %\frac{1}{\|\ve A\|_\infty^{1+1/(n-1)}}
\left(\prod_{i=1}^{n} \frac{a_i}{\|\ve A\|_\infty}\right)^{1/(n-1)}- \frac{6n^2}{T^{1/(n-1)}},
\end{split}
\end{equation*}
since $N(T)\geq (1/3)T^n$, say. Instead of summing over all $Q(T)$ in
the first summand we just consider the subset
\begin{equation*}
{\overline Q}(T)=\left\{\ve A\in Q(T) :
a_i\geq \frac{T}{2},\, 1\leq i\leq n\right \}
\end{equation*}
for which we know  $a_i/\|\ve A\|_\infty\geq 1/2$. In order to
estimate (very roughly) the cardinality  of  ${\overline Q}(T)$ we
start with $n=2$ and we denote this $2$-dimensional  set by ${\overline
  Q}_2(T)$. There are at most
\begin{equation*}
\left(\left\lfloor \frac{T}{m}\right\rfloor
-\left\lceil\frac{T}{2m}\right\rceil +1 \right)^2 \leq \left(\frac{T}{2m}+1\right)^2
\end{equation*}
tuples  $(a,b)\in [0,T]^2$ with $\gcd(a,b)=m$ and $a,b\geq T/2$. Thus
 \begin{equation*}
 \begin{split}
 \#{\overline Q}_2(T) & \geq \left(\frac{T}{2}\right)^2 -
 \sum_{m=2}^{T/2}\left(\frac{T}{2m}+1\right)^2 \\
 & \geq \frac{T^2}{4}\left(1-\sum_{m=2}^\infty \frac{1}{m^2}\right) -
 T \sum_{m=2}^{T/2}\frac{1}{m} -\frac{T}{2} \\
& \geq  \frac{T^2}{4}\left(2-\frac{\pi^2}{6}\right) - T
\sum_{m=1}^{T/2}\frac{1}{m} \\
& \geq \frac{T^2}{12} - T\, (1+\ln(T/2))  \geq    \frac{T^2}{12} -2\, T\ln(T),
 \end{split}
 \end{equation*}
for $T\geq 2$. Since $\# {\overline Q}(T)\geq  \#{\overline
  Q}_2(T)\times (T/2)^{n-2}$ we get
\begin{equation*}
\begin{split}
\frac{1}{N(T)} &\sum_{\ve A\in Q(T)} %\frac{1}{\|\ve A\|_\infty^{1+1/(n-1)}}
\left(\prod_{i=1}^{n} \frac{a_i}{\|\ve A\|_\infty}\right)^{1/(n-1)}
\geq \frac{1}{T^n} \sum_{\ve A\in \overline{Q}(T)}
%\frac{1}{\|\ve A\|_\infty^{1+1/(n-1)}}
\left(\prod_{i=1}^{n} \frac{a_i}{\|\ve A\|_\infty}\right)^{1/(n-1)} \\
&\geq \frac{\# {\overline Q}(T)}{T^n}\left(\frac{1}{2}\right)^{n/(n-1)} \geq
\left(\frac{1}{2}\right)^{n} \left(\frac{1}{12}-2\frac{\ln T
  }{T}\right)\\ &  \geq \left(\frac{1}{2}\right)^{n}\frac{1}{24},
\end{split}
\end{equation*}
for $T\geq 500$, say. Hence, all together we have found for $T\geq
500$
\begin{equation*}
\begin{split}
  \frac{1}{N(T)}\sum_{\ve A\in Q(T)} & \max_{c\in\Q^n}\frac{\gap_{c}(\ve A)}{\|\ve
  A\|_\infty^{1/(n-1)}\|\ve c\|_1}  \geq
\rho_{n-1}\left(\frac{1}{2}\right)^{n}\frac{1}{24} - 6n^2
\frac{1}{T^{1/(n-1)}} \\
&\geq
\frac{6}{500}n\left(\left(\frac{1}{2}\right)^{n}-\frac{n}{T^{1/(n-1)}}
\right)  \geq \frac{3}{500}n\left(\frac{1}{2}\right)^{n},
\end{split}
\end{equation*}
for $T\geq \max\{500, (n\, 2^{n+1})^{n-1}\}$.

\enddocument